\documentclass[12pt]{amsart}

\usepackage{amsmath} %
\usepackage{amssymb} %
\usepackage{amstext}
\usepackage{amsthm}
\usepackage[utf8]{inputenc}
\usepackage{a4wide}
\usepackage{srcltx}
\usepackage{esint}

\newcommand{\be}{\pmb e}
\newcommand{\bu}{\pmb u}

\newcommand{\beh}{\pmb f}

\newcommand{\tens}{\boldsymbol{\mathcal{S}}}
\newcommand{\sD}[1]{\boldsymbol{e}(#1)}
\newcommand{\dual}[2]{\left\langle#1,#2\right\rangle}
\newcommand{\norm}[3]{\|{#1}\|_{#2}^{#3}}

\newcommand{\ddt}{\frac{d}{dt}}
\newcommand{\dert}{\partial_t}
\renewcommand{\dh}{\mathrm{d}^h}
\renewcommand{\th}{\mathrm{\tau}^h}
\newcommand{\intom}{\int_\Omega}
\newcommand{\inthn}{\int\limits_{t_0-h}^{t_0}}
\newcommand{\finthn}{\fint\limits_{t_0-h}^{t_0}}
\newcommand{\dx}{\,dx}

\newcommand{\puniq}{p_{\rm uniq}}
\newcommand{\duniq}{\delta_{\rm uniq}}
\newcommand{\suniq}{\sigma_{\rm uniq}}
\newcommand{\buh}{\overline{\bu}_h}

\newcommand{\diver}{\operatorname{div}}
\newcommand{\rr}{\mathbb{R}}
\newcommand{\nn}{\mathbb{N}}

\newcommand{\dashint}{\fint}
\newcommand{\sleb}{semi-Lebesgue}
\newcommand{\bU}{{\pmb U}}

\newtheorem{lemma}{Lemma}[section]
\newtheorem{theorem}{Theorem}[section]
\newtheorem{definition}{Definition}[section]
\newtheorem{corollary}{Corollary}[section]
\newtheorem{remark}{Remark}[section]

\numberwithin{equation}{section}
\begin{document}

\title[Time regularity of non-Newtonian fluid]{Time regularity of flows of non-Newtonian fluids with critical power-law growth}


\author[M.~Bul\'{\i}\v{c}ek]{Miroslav Bul\'{\i}\v{c}ek}
\address{Charles University, Faculty of Mathematics and Physics, Mathematical Institute, Soko\-lov\-sk\'{a}~83, 186~75~Prague~8, Czech~Republic}
\email{mbul8060@karlin.mff.cuni.cz}

\author[P.~Kaplick\'{y}]{Petr Kaplick\'{y}}
\address{Charles University, Faculty of Mathematics and Physics, Department of Mathematical Analysis, Sokolovsk\'{a}~83, 186~75~Prague~8, Czech~Republic}
\email{kaplicky@karlin.mff.cuni.cz}

\author[D.~Pra\v{z}\'{a}k]{Dalibor Pra\v{z}\'{a}k}
\address{Charles University, Faculty of Mathematics and Physics, Department of Mathematical Analysis, Sokolovsk\'{a}~83, 186~75~Prague~8, Czech~Republic}
\email{prazak@karlin.mff.cuni.cz}


\keywords{non-Newtonian fluids, time regularity, uniqueness}
\subjclass[2000]{76A05, 76D03}

\thanks{M.B. was  supported by the Czech Science Foundation (Grant no. 16-03230S). M.B. and P.K. acknowledge membership in the Nečas Center of Mathematical Modeling (http://ncmm.karlin.mff.cuni.cz).}

\begin{abstract}We deal with the flows of non-Newtonian fluids in three dimensional setting subjected to the homogeneous Dirichlet boundary condition. Under the natural monotonicity, coercivity and growth condition on the Cauchy stress tensor expressed by a power index $p\ge 11/5$ we establish regularity properties of a solution with respect to time variable. Consequently, we can use this better information for showing the uniqueness of the solution provided that the initial data are good enough for all power--law indexes $p\ge 11/5$. Such a result was available for $p\ge 12/5$ and therefore the paper fills the gap and extends the uniqueness result to the whole range of $p$'s for which the energy equality holds.
\end{abstract}

\maketitle

\section{Introduction}

We study the generalized Navier--Stokes system
\begin{align} \label{i1}
\dert \bu + (\bu \cdot \nabla) \bu
- \diver \tens \big( \sD{\bu} \big) + \nabla \pi
	&= \beh
\\ \label{i2}
\diver \bu &= 0
\end{align}
in $Q:=(0,T) \times \Omega$ with a bounded domain $\Omega \subset \rr^3$. Here $\bu: Q\to \rr^3$ denotes the velocity field, $\beh:Q \to \rr^3$ the density of the external body forces, $\pi:Q\to \rr$ is the pressure and $\tens: \rr^{3\times 3} \to \rr^{3\times 3}$ denotes the viscous part of the Cauchy stress. The system~\eqref{i1}--\eqref{i2} is completed by the initial  and
boundary conditions
\begin{align*}
    \bu &= \pmb{0} &&\textrm{on $\partial \Omega\times (0,T)$}
\\
\bu(0) &= \bu_0 &&\textrm{in $\Omega$},
\end{align*}
We consider the usual Ladyzhenskaya-type power-law fluid introduced in \cite{La69}, i.e.,
existence of $p>2$ such that the stress tensor is a continuous nonlinear function of the symmetric
    velocity gradient~$\sD{\bu}$, satisfying for all symmetric $\be, \be_1, \be_2 \in \rr^{3 \times 3}$
\begin{equation} 	\label{i5-aS}
\begin{aligned}
    \big( \tens(\be_1) - \tens(\be_2) \big)
    : \big( \be_1 - \be_2 \big)
    &\ge
\begin{cases}
    c \big( 1 + |\be_1| + |\be_2|\big)^{p-2} |\be_1 - \be_2|^2,
    \\
    c |\be_1 - \be_2|^2 + c |\be_1 - \be_2|^p,
\end{cases}
\end{aligned}
\end{equation}
 and
\begin{equation}\label{A2}
    \big| \tens(\be)\big|\leq c \big( 1 + |\be|^{p-1}).
\end{equation}

Without going into details, the main result of the
paper can be explained as follows: the system \eqref{i1}--\eqref{i2} gives
natural a priori estimates $\bu \in L^{\infty}(0,T;L^2) \cap L^p(0,T;V_p)$.
If $p\ge 11/5$, it further follows that $\dert \bu \in
L^{p'}(0,T;V_p')$. This means that the solution becomes an
admissible test function and rigorous existence theorem can be obtained using
standard compactness and monotonicity arguments. The question of
uniqueness is however open in general in the above regularity class.

This ``existence-uniqueness gap'' is related to the fact that it is possible
to test the equation by the solution, but it is not possible to do the same
for the \emph{differences} of solutions, in view of the
nonlinear character of the problem. The key idea (due to \cite{Bekp10})
is that in the strictly subcritical case, i.e.,
\begin{equation*}
    p > \frac{11}5,
\end{equation*}
one obtains some room to estimate at least fractional differences
of solutions. Of course the critical term here is the convective term,
and due to its polynomial character,
the estimates are easily computable and can be iteratively improved to the
point where the obtained regularity finally implies uniqueness.

\par
In the present paper, we extend the result in two ways. We show that the
regularity is global, meaning up to the time $t=0$, provided that
$\bu_0 \in W^{1,p}$. Secondly, we show that
the result holds even for the critical case $p=11/5$. Here the key
ingredient is a delicate estimate involving a vector-valued version
of Gehring's lemma.

\par
The paper is organized as follows. In Section~\ref{s:FSWS}, we recall the
appropriate function spaces, including a brief review of Nikolskii
spaces. The concept of weak solution is defined, and the main result of
the paper is formulated together with its corollaries.
In Section~\ref{s:AUX}, we establish some auxiliary estimates:
we recall the standard weak-strong uniqueness, establish the
initial-time regularity, and also prove the improved integrability
in the critical case $p=11/5$ based on Gehring's' lemma.

\par
The iterative scheme, which results in the proof of the main theorem, is
explained in the final Section~\ref{s:PRF}.


\section{Bibliographical overview}

Although improving regularity in time of weak solutions is standard, see among others \cite[Theorem~III.3.5]{Temam2001}, \cite[Theorem~2.7.2]{Sohr2001}, \cite[Section~7.1]{Evans1998}, \cite[Lemma~4]{GiaGiu1973} there are not many works where this is done in the similar way as here. Perhaps the closest to our approach is the method from \cite[Section~2]{NWW1998}, where the iterative improvement of time regularity of solutions is necessary due to terms appearing when localizing equations in time. In our article the main obstacle is the convective term. The method we use is based on the same idea of iterative improvements of time regularity as the method of \cite{NWW1998} although the application is slightly different. It allows to handle problems connected with convective term and also with localization. In the case $p\in(11/5,12/5]$ there appear additional difficulties that have to be overcome. In \cite{Kap2005} this method is used to obtain full regularity of systems similar to \eqref{i1} for $p\in [2, 4)$ if $\Omega\subset\rr^2$.

Improving time regularity of weak solutions is also used in \cite{BEKP2007} to compute bounds of dimension of attractor to system \eqref{i1} if $\tens$ has potential and $p>12/5$. 

  In \cite{BK2016} a similar iterative approach was used to derive local improvement of regularity in time for strongly nondiagonal parabolic systems of p-Laplace type. The main obstacle there is not a nonlinear term similar to $K_0$ below but the terms appearing due to localization in space.  Technique of differences in time is used also in \cite{FS2015}.
  In \cite{GP2014} and \cite{FGP1} a similar approach is used to establish
time regularity and uniqueness for the Ladyzhenskaya type fluid coupled
with Cahn-Hilliard equation.

Concerning the uniqueness of solution -- already in \cite{La69}, the uniqueness is established provided $p\ge 5/2$ or in case of smooth initial condition for $p\ge 12/5$. The range $p\in [11/5, 12/5)$ however remained untouched except the case of spatial periodic condition, for which one can improve even spatial regularity, see \cite{La69,MNRR96}. Such a method is however not available for Dirichlet (or other) boundary conditions. The case of general boundary condition was firstly treated in \cite{Bekp10}, where the uniqueness in sense of trajectories\footnote{Here, in the sense of trajectories means that if $\bu_1$ and $\bu_2$ coincide for all $t\in (0,t^*)$ then they coincide also for all $t\ge t^*$.} was proven for $p>11/5$. This paper therefore completes and unifies  the uniqueness theory, i.e., for sufficiently regular initial condition, we have the global in time unique solution provided $p\ge 11/5$.


\section{Preliminaries}	\label{s:FSWS}
\subsection{Function spaces}
We employ the standard Lebesgue and Sobolev spaces,
pertinent to the weak formulation of our problem:
\begin{align*}
G &= L^2(\Omega;\rr^3) \cap \lbrace  \diver \bu = 0,\ \bu\cdot n
|_{\partial \Omega} = 0 \rbrace
	\\
V_p &= W^{1,p}(\Omega;\rr^3) \cap \lbrace  \diver \bu = 0,\
\bu |_{\partial \Omega} = 0 \rbrace
\end{align*}
Our main focus will be the time regularity of vector-valued
function $u:[0,T] \to X$, where $X$ is some Banach space.
The symbol $\ddt$ denotes the weak (distributional) derivative,
and $C$, $C^{0,\alpha}$ are continuous and $\alpha$-H\"older continuous
functions, respectively. To describe a finer scale of fractional time
regularity, we will work with the so-called Nikolskii spaces. For
$u:I \to X$, where $I \subset \rr$ is an arbitrary time interval,
and $h>0$, we set
\begin{align*}
    I_h &= \lbrace t \in I;\ t+h \in I \rbrace
    	\\
    \tau^h u(t) &= u(t+h), \quad t \in I_h
	\\
    d^h u(t) &= u(t+h) - u(t), \quad t \in I_h
\end{align*}
For $p\in[1,\infty]$ and $s\in(0,1)$, the Nikolskii space
$N^{s,p}(I;X)$ is defined via the norm
\begin{equation*}
    \norm{u}{L^p(I;X)}{} + \sup_{h>0} h^{-s} \norm{d^hu}{L^p(I_h;X)}{}
\end{equation*}
It is not difficult to see that for $s=1$, the above norm is equivalent
to $W^{1,p}(I;X)$. For a general $\sigma = k + s$, where $k\in \nn$ and
$s \in (0,1)$,
one defines $N^{\sigma,p}(I;X)$ as the space of functions with $(\ddt)^j u \in
L^p(I,X)$ for $j=0,\dots,k$ and moreover, $(\ddt)^k u \in N^{s,p}(I;X)$.

Nikolskii spaces are one instance of fractional regularity spaces:
$N^{s,p} = B^{s,p}_{\infty}$, where the latter is the Besov space.
The corresponding theory is treated in many books, e.g.\ Adams, Fournier
\cite{AF03} or Bennett, Sharpley \cite{BS88}. Relatively elementary
treatment can be found in Simon \cite{Si90}. The following embeddings are
standard, see e.g.\ \cite[Corollary 26 and 33]{Si90}.
\begin{align}	\nonumber
    N^{s,p}(I;X) &\hookrightarrow C^{0,\alpha}(I;X)
    \qquad \textrm{if $\alpha = s - \frac1p > 0 $}
    \\		\label{nimb2}
    N^{s,p}(I;X) &\hookrightarrow L^q(I;X)
    \qquad \textrm{if $\frac1q > \frac1p - s \ge 0$}
\end{align}
Nikolskii spaces are not the best choice in view of interpolation or
embedding results; note the strict condition on $q$ in \eqref{nimb2}. Their
relative advantage lies in simplicity of definition -- we will see that it is
rather straightforward to obtain estimates of $N^{s,p}$-norm. The following
special interpolation result will be useful (see Lemma~2.3 in
\cite{Bekp10} for a simple proof).

\begin{lemma}	\label{lm23}
Let $X \hookrightarrow H$, where $H$ is a Hilbert space and
$X$ is separable and dense in $H$. Then
\begin{equation*}
    N^{\alpha,p}(I;X) \cap N^{\beta,p'}(I;X')
    \hookrightarrow N^{\frac{\alpha+\beta}2,2}(I;H)
\end{equation*}
for any $\alpha$, $\beta\ge0$.
\end{lemma}

%
\subsection{Weak formulation and classical results}
We adopt the standard functional
formulation of \eqref{i1}--\eqref{i2}. Set
\begin{align*}
\dual{N(\bu)}{\psi} &= \intom \tens \big(\sD{\bu}\big) : \sD{\psi} \dx
	\\
\dual{K_0(\bu)}{\psi} &= \intom \big( \bu \otimes \bu \big) : \nabla \psi
\dx
\end{align*}
A function $\bu:[0,T] \to V_p$
will be called \emph{weak solution} if it satisfies
\begin{align} 	\label{rw1}
\bu &\in L^{\infty}(0,T;G) \cap L^p(0,T;V_p)
\end{align}
and the equation
\begin{align} 	\label{wf1}
\ddt \bu + N(\bu) = K_0(\bu) +  \beh
	\qquad \textrm{in $V_p'$}
\end{align}
holds almost everywhere in $I$.

The pressure is excluded from the weak formulation as usual. The
critical condition~$p\ge 11/5$ means that the derivative
belongs to the corresponding dual space
\begin{align}	\label{rdt1}
	\ddt \bu &\in L^{p'}(0,T;V_p').
\end{align}
More precisely, one has the following estimate.
\begin{lemma}	\label{lm:dual1}
Let $p \ge 11/5$ and $\tens$ satisfy \eqref{A2}. Then the weak solution satisfies for almost every $t\in (0,T)$
\begin{align}	\label{Ph1}
\norm{\ddt \bu(t)}{V_p'}{} &\le C\big( 1 +\norm{\bu(t)}{V_p}{p-1}+\norm{\beh(t)}{V'_p}{} \big)
\end{align}
where $C$ possibly depends on the (essentially bounded) function
$\norm{\bu(t)}{2}{}$. 
\end{lemma}
\begin{proof}
Omitting the variable $t$ for simplicity, we take $\psi \in V_p$
with $\norm{\psi}{}{} \le 1$ in \eqref{wf1} and estimate
\[
\dual{\ddt \bu}{\psi} = - \dual{N(\bu)}{\psi} +
\dual{K_0(\bu)}{\psi} +\dual{\beh}{\psi}= D_1 + D_2+D_3.
\]
Clearly, it follows from \eqref{A2} that  $|D_1| +|D_3|\le C ( 1 + \norm{\bu}{V_p}{p-1}+\norm{\beh}{V_p'}{})$. Using the
interpolation (recall $\Omega\subset \rr^3$)
\begin{equation}	\label{ip0}
\norm{v}{2p'}{} \le \norm{v}{2}{1-a} \norm{v}{1,p}{a}
\qquad a = \frac{3}{5p-6},\ 1-a = \frac{5p-9}{5p-6},
\end{equation}
which is valid all $p\ge 9/5$, we have
\begin{align*}
|D_2| 	\le \intom |\bu|^2 |\nabla \psi| \dx
	\le \norm{\bu}{2p'}{2} \norm{\nabla \psi}{p}{}
	\le \norm{\bu}{2}{2(1-a)} \norm{\bu}{V_p}{2a}
\norm{\psi}{V_p}{} \le C \big( 1 + \norm{\bu}{V_p}{p-1} \big).
\end{align*}
Here we have used the Young inequality and the fact that $2a \le p - 1$, which is just $p\ge 11/5$.
\end{proof}

Next, due to the monotonicity of $\tens$ and assumption on $p$, we have the following existence result.
\begin{lemma}
Let $p\ge 11/5$, $\beh \in L^{p'}(0,T;V_p')$ and $\bu_0 \in G$.  Then there
exists at least one weak solution within the class \eqref{rw1}, \eqref{rdt1},
satisfying $\bu(0)=\bu_0$.
\end{lemma}
\begin{proof}

Let us outline the formal a priori estimates.
Apply \eqref{wf1} to $\bu$.
Thanks to \eqref{i2} and the fact that $\bu$ vanishes on $\partial \Omega$, the convective term (the first term on the right hand side of \eqref{wf1}) disappears and we obtain the energy identity
\begin{equation}	\label{ei}
    \frac12\ddt \norm{\bu}{2}{2}
    + \intom \tens(\sD{\bu}) : \sD{\bu} \,dx
    =\dual{\beh}{\bu}.
\end{equation}
In view of \eqref{i5-aS}, Korn's and Poincar\'e's inequalities, we have
\begin{equation}\label{A1}
\begin{aligned}
    \intom \tens(\sD{\bu}) : \sD{\bu} \,dx
    &\ge c \big( \norm{\bu}{1,2}{2} + \norm{\bu}{1,p}{p} \big)
    \\
    \dual{\beh}{\bu}
    & \le \epsilon \norm{\bu}{1,p}{p} + C_\epsilon
    \norm{\beh}{V_p'}{p'}
\end{aligned}
\end{equation}
whence the estimate \eqref{rw1} follows easily. Secondly, by
Lemma~\ref{lm:dual1}, one has \eqref{rdt1}, and $\bu$ is indeed an
admissible test function. With a suitable approximating scheme,
the above circle of reasoning can be turned into a rigorous existence theorem, employing the usual compactness and monotonicity argument to pass to the limit in nonlinear term $K_0(\bu)$.
We omit further details, referring e.g. to \cite[Chapter 5]{MNRR96}.
\par
Note that it also follows that $\bu$ has a continuous representative in $C([0,T];G)$ and the initial condition $\bu(0)=\bu_0$ makes sense.
\end{proof}

The last classical result, we recall here (see e.g. \cite{MNRR96}), is the ``fundamental'' difference inequality that can be further used for proving the weak-strong uniqueness result.
\begin{lemma}	\label{lm:diff}
Let $\bu_1$, $\bu_2$ be weak solutions corresponding to right hand side functions $\beh_1$ and $\beh_2$ respectively and let $p\ge 11/5$. Let us define
\begin{equation} 	\label{p:uniq}
\puniq := \frac{2p}{2p-3}.
\end{equation}
Then
\begin{equation}\label{eq:unique}
\begin{split}
&\ddt  \norm{\bu_1 - \bu_2}{2}{2}
+
c \big( \norm{\bu_1 - \bu_2}{V_2}{2} + \norm{\bu_1 -
\bu_2}{V_p}{p} \big)
\\
&\qquad \le C\norm{\bu_2}{V_p}{\puniq} \norm{\bu_1 - \bu_2}{2}{2}
+C\norm{\beh_1-\beh_2}{V_p'}{p'},
\end{split}
\end{equation}
where the constant $C$ depends only on $\Omega$ and $p$.
\end{lemma}

\begin{proof}
Since $p\ge 11/5$, we can use $\bu_1-\bu_2$ as a test function in \eqref{wf1} for $\bu_1$ and $\bu_2$, respectively and observe
\begin{align*}
&\frac12 \ddt \norm{\bu_1 - \bu_2}{2}{2} = \dual{\ddt ( \bu_1 -
\bu_2 ) }{\bu_1 - \bu_2}
\\
&= - \dual{N(\bu_1) - N(\bu_2)}{\bu_1 - \bu_2}
+ \dual{K_0(\bu_1) - K_0(\bu_2)}{\bu_1 - \bu_2}
+ \dual{\beh_1 - \beh_2}{\bu_1 - \bu_2}
\\
&=: D_1 + D_2 +D_3.
\end{align*}
By the $p$-ellipticity of $N(\bu)$, i.e., the assumption \eqref{i5-aS}, we have
\[
D_1 \le - c \big( \norm{\bu_1 - \bu_2}{V_2}{2} + \norm{\bu_1 -
\bu_2}{V_p}{p}  \big).
\]
Using the interpolation (valid for $p\ge 3/2$)
\begin{equation*}
\norm{v}{2p'}{} \le C\norm{v}{2}{\frac{2p-3}{2p}} \norm{v}{1,2}{\frac{3}{2p}},
\end{equation*}
integration by parts,  \eqref{i2} and the Poincar\'{e} inequality, we have
\begin{align*}
D_2&=\int_{\Omega} (\bu_1\otimes \bu_1 - \bu_2\otimes \bu_2)\cdot \nabla(\bu_1 - \bu_2)\, dx=\int_{\Omega} (\bu_2\otimes (\bu_1 - \bu_2))\cdot \nabla(\bu_1 - \bu_2)\, dx\\
&=-\int_{\Omega} (\nabla \bu_2 \cdot ((\bu_1 - \bu_2)\otimes(\bu_1 - \bu_2))\, dx\le  \norm{\bu_2}{V_p}{} \norm{\bu_1 - \bu_2}{2p'}{2}\\
&\le \norm{\bu_2}{V_p}{} \norm{\bu_1 - \bu_2}{2}{\frac{2p-3}{p}}
	\norm{\bu_1 - \bu_2}{V_2}{\frac{3}{p}}
\le
\epsilon \norm{\bu_1 - \bu_2}{V_2}{2} + C_\epsilon
\norm{\bu_2}{V_p}{\puniq} \norm{\bu_1 - \bu_2}{2}{2}.
\end{align*}
The estimate of $D_3$ is straightforward.
Summarizing these estimates with $\epsilon>0$ small enough finishes the
proof.
\end{proof}
We see that if at least one weak solution belongs to $L^{\puniq}(0,T;V_p)$, one can apply the Gronwall inequality to \eqref{eq:unique} to conclude the continuous dependence on data and/or the uniqueness of solution. We also recall that
for $p\ge 5/2$, one has $p\ge \puniq$, and consequently uniqueness holds true
in the class of weak solutions.

\subsection{Main result}
Let us now formulate our main result.

\begin{theorem}	\label{th:main}
Let $p \geq 11/5$, $\beh \in L^{p'}(0,T;V_p')$ and $\bu_0 \in G$.
In case that $p<5/2$, assume moreover $\beh \in N^{\delta,p'}(0,T;V_p')$ with some $\delta > \duniq$, where
\begin{equation}\label{duniq}
\duniq := (p-1)\left(\frac{5}{2p} - 1\right).
\end{equation}
Then for an arbitrary weak solution $\bu$ the following holds true:
\begin{enumerate}
\item
For any $t_0 \in (0,T)$ we have $\bu \in L^{\puniq}(t_0,T;V_p)$.
\item
  If $\bu_0\in V_p$, the conclusion holds for $t_0=0$ as well.
\end{enumerate}
\end{theorem}

Recalling that $\puniq$ is the critical integrability condition that
enables one to handle the equation for \emph{difference} of two solutions
(see Lemma~\ref{lm:diff} above) we also obtain:

\begin{corollary} \label{cor:uniq}
Let the conditions of Theorem~\ref{th:main} hold.
Let $\bu_1$, $\bu_2$ be two weak solutions that coincide
on some $[0,\tau]$, $\tau > 0$. Then $\bu_1 = \bu_2$ on $[0,T]$.
The same conclusion holds if $\bu_1(0)=\bu_2(0) \in V_p$.
\end{corollary}

\begin{corollary} \label{cor:maxr}
  Let $p \geq 11/5$ and $\beh \in N^{1/p',p'}(0,T;V_p')$, $u_0\in G$, $\bu$ be a weak solution, and $t_0 \in (0,T)$ be arbitrary. Then
\begin{equation*}
	\bu \in N^{\frac12,\infty}(t_0,T;L^2) \cap N^{\frac1p,p}(t_0,T;V_p) \cap N^{\frac12,2}(t_0,T;V_2)
\end{equation*}
The same conclusion holds for $t_0=0$ provided that $\bu_0\in V_p$ and $\sup_{h\in(0,T)}\dashint_{0}^{h}\norm{\beh}{V_p'}{p'}<+\infty$.
\end{corollary}

\section{Auxiliary estimates}	\label{s:AUX}

In this section we summarize several auxiliary estimates. We establish the ``initial time regularity'' provided
$\bu_0$ belongs to $V_p$, see Lemma~\ref{lm:w2}, which is a starting
point of our iteration scheme.
Finally, we show how the integrability of $\bu$ can be improved by reverse H\"older inequality with increasing support. This allows us later to prove the main theorem also for $p=11/5$.

\begin{definition}
  Let $\beh:(0,T)\to V_p'$. We call $t\in[0,T)$ \sleb\ point of $\beh$ if $\sup_{h\in(0,T-t)}\dashint_{t}^{t+h}\norm{\beh}{V_p'}{p'}$ is finite.
\end{definition}

\begin{lemma} \label{lm:w2}
Let $\bu$ be the representative of a weak solution continuous with values in $L^2(\Omega)$,  $p\ge 11/5$ and $\beh \in N^{\delta,p'}(0,T; V_p')$ with $1/p'\ge \delta > 2\tau/p'\ge 0$. If $t_0 \in [0,T)$  and $\bu(t_0)\in V_p$ then
\begin{equation} 	\label{nik:t00}
\exists c>0, \forall h\in (0,T-t_0):\norm{\bu(t_0+h) - \bu(t_0)}{2}{2} \le ch^{2\tau}.
\end{equation}
If in addition $t_0$ is a \sleb\ point of $\beh$ then
\begin{equation} 	\label{nik:t0}
\exists c>0, \forall h\in (0,T-t_0):\norm{\bu(t_0+h) - \bu(t_0)}{2}{2} \le ch.
\end{equation}

\end{lemma}


\begin{proof}
Recall that $\bu$ stands for the representative continuous in $[0,T]$ with values in $L^2(\Omega)$.
  We can write
\begin{equation*}
\begin{aligned}
\norm{\bu(t_0+h) - \bu(t_0)}{2}{2}
 &= \big( \bu(t_0+h) - \bu(t_0) , \bu(t_0+h) - \bu(t_0) \big)
\\
&= \big( \bu(t_0+h) - \bu(t_0) , ( \bu(t_0+h) + \bu(t_0)) - 2 \bu(t_0) \big)
\\
&= \norm{\bu(t_0+h)}{2}{2} - \norm{\bu(t_0)}{2}{2}
	- 2 \big( \bu(t_0+h) - \bu(t_0) , \bu(t_0) \big)
\\
 &= I_1 + I_2.
\end{aligned}
\end{equation*}
We start estimating $I_1$. Using the energy equality \eqref{ei} and \eqref{A1} we get
$$
\begin{aligned}
I_1=\int_{t_0}^{t_0+h}&(-\dual{N(\bu(s))}{\bu(s)}+\dual{h(s)}{\bu(s)})\,ds\\ \leq
&\int_{t_0}^{t_0+h}(-c \big( \norm{\bu(s)}{1,2}{2} + \norm{\bu(s)}{1,p}{p} \big)+C\norm{\beh(s)}{V_p'}{p'})\,ds.
\end{aligned}
$$

The term $I_2$ can be rewritten  as time derivative of $u(t)$, i.e., we have
$$
I_2=-2\int_{t_0}^{t_0+h}\dual{\frac{d}{dt}\bu(s)}{\bu(t_0)}\, ds.
$$

Then we use \eqref{Ph1} and finally Young's inequality with $\epsilon>0$ to obtain
$$
\begin{aligned}
  I_2&\leq C\int_{t_0}^{t_0+h} \big( 1 + \norm{\bu(s)}{V_p}{p-1}+\norm{\beh(s)}{V_p'}{} \big)\norm{\bu(t_0)}{V_p}{}\,ds\\
&\leq \int_{t_0}^{t_o+h}\epsilon \norm{\bu(s)}{V_p}{p}+ C(1+\norm{\bu(t_0)}{V_p}{p}+\norm{\beh(s)}{V_p'}{p'})\,ds.
\end{aligned}
$$

If we combine the estimates of $I_1$ and $I_2$ and choose $\epsilon>0$ sufficiently small we get
$$
\begin{aligned}
  &\norm{\bu(t_0+h) - \bu(t_0)}{2}{2}\\
  &\phantom{M}\leq \int_{t_0}^{t_0+h}(-\frac c2 \big( \norm{\bu(s)}{1,2}{2} + \norm{\bu(s)}{1,p}{p} \big)ds+Ch\big(\dashint_{t_0}^{t_0+h}\norm{\beh(s)}{V_p'}{p'}+1+\norm{\bu(t_0)}{V_p}{p}ds\big)\\
  &\phantom{M}\leq
  Ch\dashint_{t_0}^{t_0+h}\norm{\beh(s)}{V_p'}{p'}+1+\norm{\bu(t_0)}{V_p}{p}\,ds=Ch(1+\norm{\bu(t_0)}{V_p}{p})+C\int_{t_0}^{t_0+h}\norm{\beh(s)}{V_p'}{p'}\,ds.
\end{aligned}
$$

Consequently, if $t_0$ is a \sleb\ point of $\beh:(0,T)\to V_p'$ and $\bu(t_0)\in V_p$ we get \eqref{nik:t0}. Similarly, if $\beh \in N^{\delta,p'}(0,T; V_p')$ with $1/p'>\delta > 2\tau/p'$ then by embedding theorem, we get that $\beh \in L^{p'/(1-2\tau)}(0,T; V_{p}')$. Thus, by H\"{o}lder's inequality, we see that
$$
\int_{t_0}^{t_0+h}\norm{\beh(s)}{V_p'}{p'}\,ds\le \left(\int_{t_0}^{t_0+h}\norm{\beh(s)}{V_p'}{\frac{p'}{1-2\tau}}\,ds\right)^{1-2\tau} h^{2\tau}.
$$
Hence, altogether we have
$$
\begin{aligned}
  &\norm{\bu(t_0+h) - \bu(t_0)}{2}{2}\leq C(h+h^{2\tau})
\end{aligned}
$$
and \eqref{nik:t00} follows.
\end{proof}

\medskip
\begin{remark}
Since a weak solution has a representative continuous with values in $L^2(\Omega)$ that satisfies $\bu(0)=\bu_0$, the statement \eqref{nik:t0} holds provided $\bu_0\in V_p$ and $0$ is \sleb\ point of $\beh$.
\end{remark}

\medskip
\begin{remark}

Note that it follows from \eqref{rw1}, \eqref{rdt1} and Lemma~\ref{lm23}
(with $\alpha=1/p$, $\beta=1+1/p'$) that $\bu \in N^{1/2,2}(0,T;G)$, i.e.,
\begin{equation*}
\int\limits_{t_0}^{T-h} \norm{\dh \bu}22 \le C h \,;
\end{equation*}
in \eqref{nik:t0} we get this estimate, so to say, pointwise.
\end{remark}


\begin{lemma}\label{lm:GHR}
Let $p\ge 11/5$, let $\beh \in L^{q_0}(0,T;V_p')$ for some
$q_0 > p'$. Then there is $q>p$ such that
$\bu \in L^q_{\rm loc}(0,T;V_p)$. Moreover, if $\bu_0 \in V_p$,
the conclusion holds globally.
\end{lemma}

\begin{proof}
First we concentrate on the local regularity result.
  We show that there exist $C>0$ such that
  for any $t_0 \in (0,T)$ and $h\in(0,t_0)$
  \begin{equation}	\label{preG}
\left( \fint_{t_0-h/2}^{t_0} \norm{\bu(t)}{V_p}{p}\,dt \right)^\frac1p
\le
C \left( 1 + \left( \fint_{t_0-h}^{t_0} \norm{\bu(t)}{V_p}{p-1}\,dt
	\right)^\frac{1}{p-1}
	+ \left( \fint_{t_0-h}^{t_0} \norm{\beh(t)}{V_p'}{\frac{p}{p-1}}\,dt \right)^\frac1{p}
\right).
\end{equation}
The conclusion then follows by an application of a variant of Gehring's lemma with increasing support, see e.g.  \cite[Proposition~V.1.1]{Gia83}. Here, the support grows only on one side of the interval $(t_0-h/2,t_0)$, yet the situation
can easily be accommodated according to \cite[Proposition~1.3]{GiaStru82}.

We fix $\bU\in V_p$
and test the equation \eqref{wf1} by $\bu(t) - \bU$ to obtain (using $\dual{K_0(\bu(t))}{\bu(t)}=0$)
\[
\frac12 \ddt \norm{\bu(t) - \bU}{2}{2}
+ \dual{N(\bu(t))}{\bu(t)} = \dual{K_0(\bu(t)) + N(\bu(t))}{\bU}
+ \dual{\beh(t)}{\bu(t) - \bU}.
\]
In the standard way we estimate (using the assumptions \eqref{i5-aS}--\eqref{A2})
\begin{align*}
\ddt \norm{\bu(t) - \bU}{2}{2} + \alpha \norm{\bu(t)}{V_p}{p}
\le C \Big( 1 +& \norm{\bu(t)}{2p'}{2} \norm{\bU}{V_p}{}
	+ \norm{\bu(t)}{V_p}{p-1} \norm{\bU}{V_p}{}
	\\
	&+ \norm{\beh(t)}{V_p'}{}( \norm{\bu(t)}{V_p}{} + \norm{\bU}{V_p}{} )
	\Big).
\end{align*}
 Here and in what follows, $C>0$ and $\alpha>0$ are generic constants that may change from line to line and  depend only on the data of the equation.

Invoking now the interpolation \eqref{ip0}, a~priori estimate \eqref{rw1} and Young's inequality, we proceed to
\begin{align*}
\norm{\bu(t)}{2p'}{2} \norm{\bU}{V_p}{} &\le
	c \norm{\bu(t)}{2}{\frac{2(5p-9)}{5p-6}}
		\norm{\bu(t)}{V_p}{\frac{6}{5p-6}}
		\norm{\bU}{V_p}{}
		\le C \norm{\bu(t)}{V_p}{\frac{6}{5p-6}} \norm{\bU}{V_p}{}
		\\
	&\le \epsilon \norm{\bu(t)}{V_p}{p}
		+ C_\epsilon \big( \norm{\bU}{V_p}{p} + 1\big)
\end{align*}
as by $p\ge 11/5$ we just have $6/[p(5p-6)] + 1/p \le 1$. Thus, we
arrive at
\begin{equation}	\label{start2}
 \ddt \norm{\bu(t) - \bU}{2}{2}
+ \alpha \norm{\bu(t)}{V_p}{p} \le C \big(
	1 + \norm{\bU}{V_p}{p} + \norm{\beh(t)}{V_p'}{p'}
		\big),
\end{equation}
which is the basis for a further investigation.

If $t_0 \in (0,T)$ and $h\in(0,t_0)$ we set
\[
	\bU=\buh := \fint_{t_0-h}^{t_0} \bu(t)\,dt.
\]

We multiply \eqref{start2} by $\xi(t) = (t - (t_0-h))/h^2$,  integrate over
$t \in (t_0-h,t_0)$ and after a simple manipulation and using the fact that $\xi(t_0-h)=0$, we  obtain
\begin{equation}	\label{geh0}
\fint_{t_0 - h/2}^{t_0}
	\norm{\bu(t)}{V_p}{p}\, dt
	\le C \left( 1 + \norm{\buh}{V_p}{p}+ \finthn  \frac{\norm{\bu(t) - \buh}{2}{2}}{h} + \norm{\beh(t)}{V_p'}{p'}\, dt
	\right).
\end{equation}
Observing that
\begin{equation*}
	\norm{\buh}{V_p}{p} \le \left( \finthn \norm{\bu(t)}{V_p}{}\, dt
				\right)^{p}
\end{equation*}
it only remains to treat the second term on the right hand side of \eqref{geh0} to obtain \eqref{preG}. Towards this end,
note first that the identity
\begin{equation*}
\bu(t) - \buh = \frac1h \inthn u(t) - u(s)\,ds
	= \frac1h \inthn \int_{s}^{t} \frac{d}{d\tau} \bu(\tau)\,d\tau \,ds
\end{equation*}
holds in $V_p'$. On the other hand $\bu(t) - \buh \in V_p$ for almost all $t$ and therefore
\begin{equation*}
\norm{\bu(t) - \buh}{2}{2} = \dual{\bu(t) - \buh}{\bu(t) - \buh}.
\end{equation*}
Consequently, we have
\begin{equation*}
\finthn  \frac{\norm{\bu(t) - \buh}{2}{2}}{h} \, dt
= \frac1{h^3} \inthn \inthn \int_{s}^{t}
	\dual{\frac{d}{d\tau} \bu(\tau)}{\bu(t)-\buh}\,d\tau \, ds \, dt\,.
\end{equation*}
Invoking now this equation together with Lemma~\ref{lm:dual1}, we estimate the term on the right hand side further as
\begin{align*}
\finthn  \frac{\norm{\bu(t) - \buh}{2}{2}}{h} \, dt &\le \frac{c}{h^2} \inthn \inthn
		\big( 1 + \norm{\bu(\tau)}{V_p}{p-1} \big)
		\big( \norm{\bu(t)}{V_p}{} + \norm{\buh}{V_p}{} \big)
		\,d\tau \,dt
\\
	&= c \big( 1 + \finthn \norm{\bu(\tau)}{V_p}{p-1} \,d\tau\big)
		\finthn  \big(
	\norm{\bu(t)}{V_p}{} + \norm{\buh}{V_p}{} \big) \,dt
\\
	&\le C \left( 1 +
	\left( \finthn \norm{\bu(\tau)}{V_p}{p-1} \, d\tau
		\right)^{\frac{p}{p-1}}
	+ \left(  \finthn \norm{\bu(t)}{V_p}{} \, d t
                                \right)^{p}
		\right),
\end{align*}
where we used the Young inequality for the last estimate. We see that \eqref{preG} holds for $t_0 \in (0,T)$ and $h\in(0,t_0)$ and the local regularity result follows by Gehring's lemma.

To prove the global improvement of regularity we extend $\bu$ by $0$ to $t<0$ and show that if $\bu_0\in V_p$ the inequality \eqref{preG} holds for any $t_0 <T$ and $h> 0$. The situation $t_0\in(0,T)$, $h\in(0,t_0)$ was treated in the previous part of the proof.
Now we consider $t_0\in(0,T)$, $h>t_0$. We set $\bU=\bu_0$ in \eqref{start2} to get for $t\in(0,T)$
\begin{equation*}
  \ddt \norm{\bu(t) - \bu_0}{2}{2}+ \alpha \norm{\bu(t)}{V_p}{p}
  \le C \big(1 + \norm{\bu_0}{V_p}{p} + \norm{\beh(t)}{V_p'}{p'}\big)
  \le C\big(1 + \norm{\beh(t)}{V_p'}{p'}\big).
\end{equation*}
Integrating this inequality from $0$ to $t_0$ we get
\begin{equation*}	
 \int_0^{t_0}\norm{\bu(t)}{V_p}{p}d t
  \le C\big(t_0 + \int_{0}^{t_0}\norm{\beh(t)}{V_p'}{p'}d t\big).
\end{equation*}

Further we compute
\begin{equation*}
\begin{split}
  \fint_{t_0-h/2}^{t_0}\|\bu(t)\|_{V_p}^p\,dt
  \le \frac Ch\int_{0}^{t_0}\|\bu(t)\|_{V_p}^p\,dt
  &\le \frac{C}{h}(t_0+\int_{0}^{t_0}\|\beh(t)\|^{p'}_{V_p^*}\,dt)\\
  &\le C(1+\fint_{t_0-h}^{t_0}\|\beh(t)\|^{p'}_{V_p^*}\,dt).
  \end{split}
\end{equation*}

Since estimate \eqref{preG} clearly holds also if $t_0<0$ we finally get that under the assumption $u_0\in V_p$ the inequality \eqref{preG} holds for any $t_0<T$, $h>0$. Consequently, we get the global improvement of regularity of $\bu$ by Gehring's lemma.
\end{proof}

\section{Proof of the main theorem}	\label{s:PRF}


This section is devoted to the proof of what we formulate as the
main result: Theorem~\ref{th:main}. It seems convenient to split
the idea into two auxiliary lemmas.

In Lemma~\ref{lm:key1}, we show that the
Ladyzhenskaya fluid  -- without the convective term -- reflects
the time regularity of the right-hand side in the class of Nikolskii
spaces, provided the initial time regularity condition \eqref{ict0}
holds.  This can be seen as a
generalization of a well-known fact that the
$L^{\infty}(0,T;G) \cap L^p(0,T;V_p)$
norm of the solution is estimated by the $L^{p'}(0,T;V_p')$ norm
of the right-hand side and the $L^2$ norm of initial condition $\bu_0$.

\par
Lemma~\ref{lm:key2} then focuses on the convective
term $K_0(\bu)$. It shows that if
$\bu \in N^{\tau,\infty}(t_0,T;G) \cap N^{\sigma,p}(t_0,T;V_p)$,
then $K_0(\bu) \in N^{\delta,p'}(t_0,T;V_p)$ for suitable
$\delta>\sigma$ depending on $\tau$ and $\sigma$, provided the initial
time regularity at $t=t_0$ is satisfied, cf. \eqref{ict0}.
This generalizes another well-known
fact, namely that $K_0(\cdot)$ is bounded from
$L^{\infty}(0,T;G)\cap L^p(0,T;V_p)$ into its dual if $p\ge 11/5$.

\begin{lemma} \label{lm:key1}
Let $t_0\in[0,T)$, $\delta\in(0,1)$ and let $\bu \in L^p(t_0,T;V_p)$ satisfy
\begin{equation}\label{eq:key1}
\ddt \bu + N(\bu) =H \qquad \textrm{in }V_p'
\end{equation}
almost everywhere in $(t_0, T)$, where $H\in N^{\delta,p'}(t_0,T;V_p')$. Let us define
\begin{equation*}
    \tau = \frac{ \delta p }{2(p-1)}
    \qquad
    \sigma = \frac{\delta}{p-1}
\end{equation*}
and assume that $h_0\in(0,T-t_0)$ satisfies
\begin{equation}	\label{ict0}
\norm{\bu(t_0+h) - \bu(t_0)}{2}{2} \le c_1 h^{2\tau}
\qquad \textrm{for $h\in(0,h_0)$.}
\end{equation}
Then $\bu \in N^{\tau,\infty}(t_0,T;G) \cap
N^{\sigma,p}(t_0,T;V_p)$.
\end{lemma}
\begin{remark}
  Later we will always assume that $\sigma<1/2$ (and $\tau<1/p$).
\end{remark}

\begin{proof} Applying $\dh$ to \eqref{eq:key1} and testing the result  by $\dh
\bu$, one obtains
\[
\frac12 \ddt \norm{\dh \bu}{2}{2}
+ \dual{\dh N(\bu)}{\dh \bu} = \dual{\dh H(t)}{\dh \bu}.
\]
Here
\[
\dual{\dh N(\bu)}{\dh \bu} \ge c \big( \norm{\dh \bu}{V_p}{p} +
\norm{\dh \bu}{V_2}{2} \big)
\]
in view of the $p$-ellipticity of $N$, i.e., \eqref{i5-aS}. Further, with the help of the Young inequality, we deduce that
\[
\dual{\dh H(t)}{\dh \bu}
\le \norm{\dh H(t)}{V_p'}{} \norm{\dh \bu}{V_p}{}
\le \frac{c}2 \norm{\dh \bu}{V_p}{p} + C \norm{\dh
H(t)}{V_p'}{p'}
\]
and finally
\begin{equation*}
\sup_{t_0 \le t \le T-h} \norm{\dh \bu(t)}{2}{2}
+ c \int_{t_0}^{T-h} \big( \norm{\dh \bu}{V_p}{p} +
\norm{\dh \bu}{V_2}{2} \big) \,dt
\le c_1 h^{2\tau} + C \int_{t_0}^{T-h}  \norm{\dh H(t)}{V_p'}{p'} \,dt.
\end{equation*}
The last term is estimated by $c h^{\delta p'}=Ch^{2\tau}$ and the
conclusion follows.
\end{proof}

Since the embedding theorem for Nikolskii spaces is not sharp
(cf.\ \eqref{nimb2}), we will repeatedly write $a+\epsilon$
or $a-\epsilon$ for some number strictly larger or smaller than $a$,
respectively; the value $\epsilon>0$ will be arbitrarily small and
its values can change from line to line.
Hence we have $N^{s,p}(0,T) \subset L^{q-\epsilon}(0,T)$, where $1/q = 1/p - s$ whenever $s<1/p$.

\begin{lemma}	\label{lm:key2}
Let $p \ge 11/5$ and  $\bu \in N^{\frac12, 2}(t_0,T; G)\cap N^{\tau,\infty}(t_0,T;G) \cap
N^{\sigma,p}(t_0,T;V_p)$ with some $\tau \in [0,1/2]$ and $\sigma \in [0,1/p]$.
Then $K_0(\bu) \in
N^{\delta-\epsilon,p'}(t_0,T;V_p')$, where
\begin{equation}	\label{conc:k2}
\delta =\left\{
\begin{aligned}
    &\frac{5p-11}{5p-6} + \frac{6\sigma}{5p-6}+\frac{\tau(-5p+13-6\sigma)}{5p-6} &&\textrm{if } 5p-13+6\sigma< 0,\\
&\frac{5p-9}{2(5p-6)}+\frac{3\sigma}{5p-6}, &&\textrm{if } 5p-13+6\sigma\ge 0.
\end{aligned}
\right.
\end{equation}

If $\tau=\sigma=0$, then $K_0(\bu) \in N^{\delta,p'}(t_0,T;V_p')$ with $\delta=\frac{5p-11}{5p-6}$ precisely.
\end{lemma}

\begin{proof}
Let $\psi \in L^p(t_0,T;V_p)$ with $\norm{\psi}{}{} \le 1$, $h\in (0,T-t_0)$.
We set $T_h=T-h$ and estimate
\begin{align*}
\left| \int_{t_0}^{T_h}\dual{\dh K_0(\bu) }{\psi}\, dt\right|
\le \int_{t_0}^{T_h}\int_{\Omega} |\dh (\bu\otimes \bu)|  |\nabla \psi|
\, dx\,dt
\le 2\int_{t_0}^{T_h} \norm{\dh \bu}{2p'}{} \norm{\bu}{2p'}{}
\norm{\nabla \psi}{p}{} \, dt.
\end{align*}
We use the interpolation \eqref{ip0} (and keep value for $a$ from this)
to further estimate
\begin{equation}
\begin{aligned} \label{est-K-1}
&\le c \int_{t_0}^{T_h} \norm{\dh \bu}{2}{1-a} \norm{\dh \bu}{V_p}{a}
\norm{\bu}{2}{1-a} \norm{\bu}{V_p}{a} \norm{\psi}{V_p}{} \,dt\\
&\le C \left( \int_{t_0}^{T_h} \norm{\dh \bu}{2}{\widetilde{P}(1-a)}\,dt
\right)^{\frac{1}{\widetilde{P}}}
\left( \int_{t_0}^T \norm{\dh \bu}{V_p}{p}\,dt
\right)^{\frac{a}{p}}\|\bu\|_{L^{\infty}(0,T;G)}^{1-a}\left( \int_{t_0}^T \norm{\bu}{V_p}{p_{\sigma}}\,dt
\right)^{\frac{a}{p_{\sigma}}},
\end{aligned}
\end{equation}
where we used  H\"older's inequality
with the exponents $\widetilde{P}$, $p/a$, $\infty$, $p_{\sigma} /a$
and $p$. Here, $p_{\sigma}$ is such that $N^{\sigma,p} \subset
L^{p_{\sigma}}$, i.e., it is given by
\begin{align*}
\frac{1}{p_{\sigma}} &= \frac1p - \sigma+\epsilon
\end{align*}
with an arbitrary small $\epsilon>0$ for $\sigma >0$ and $\epsilon=0$ if $\sigma=0$. The number  $\widetilde{P}$ is computed from the H\"older's
condition, hence
\begin{equation*}
\frac{1}{\widetilde{P}} = 1-\frac{1}{p}-\frac{a}{p_{\sigma}}-\frac{a}{p}.
\end{equation*}
Inserting the value of $a$ from \eqref{ip0} we get
\begin{align*}
\frac{1}{\widetilde{P}} &= \frac{5p - 11}{5p-6} + (\sigma-\epsilon) a.
\end{align*}
Hence, using the assumptions on $\bu$, we can continue in estimating of \eqref{est-K-1} as
\begin{align} \label{est-K0}
&\le C \left( \int_{t_0}^{T_h} \norm{\dh \bu}{2}{\widetilde{P}(1-a)}\,dt
\right)^{\frac{1}{\widetilde{P}}}
\left( \int_{t_0}^T \norm{\dh \bu}{V_p}{p}\,dt
\right)^{\frac{a}{p}}.
\end{align}
Finally, we distinguish two cases. If
$$
\widetilde{P}(1-a)= \frac{5p-9}{5p - 11 + 3(\sigma-\epsilon) }\le 2 \quad \Longleftrightarrow \quad \frac{-5p+13}{6}\le \sigma -\epsilon
$$
we use the H\"{o}lder inequality on the first term to obtain
%
\begin{align*}
\le C \left( \int_{t_0}^{T_h} \norm{\dh \bu}{2}{2}\,dt
	\right)^\frac{(1-a)}{2} \left( \int_{t_0}^T \norm{\dh \bu}{V_p}{p}\,dt
\right)^{\frac{a}{p}}
\le c h^{\frac{(1-a)}{2}+a\sigma},
\end{align*}
which, after using the definition of $a$, leads to the second part of \eqref{conc:k2}.

In case that $\widetilde{P}(1-a)>2$, we interpolate the first term in \eqref{est-K0} into $L^2(0,T)$ and $L^{\infty}(0,T)$, which gives
\begin{align*}
\le C
\left( \int_{t_0}^{T_h} \norm{\dh \bu}{2}{2}\,dt
        \right)^\frac1{\widetilde{P}}
\left(  \int_{t_0}^{T_h} \norm{\dh \bu}{V_p}{p}\,dt
\right)^\frac{a}{p} \norm{\dh \bu}{L^{\infty}(t_0,T-h;G)}{b}\le c h^{ \frac{1}{\widetilde{P}} + a\sigma + b\tau},
\end{align*}
where for the second inequality we used the assumption on $\bu$ and defined
\begin{equation*}
    b :=  1-a - \frac{2}{\widetilde{P}}=\frac{-5p+13-6(\sigma-\epsilon)}{5p-6}.
\end{equation*}
This then clearly gives the first part of \eqref{conc:k2}.
\end{proof}

\begin{proof}[Proof of Theorem~\ref{th:main}]
In case $p\ge 5/2$, then $\puniq \le p$ and there is nothing to prove. Hence, we will consider only the case $11/5 \le p < 5/2$ and prove that under the assumptions of Theorem~\ref{th:main} we have $\bu \in L^{\puniq}(t_0,T;V_p)$. By the
embedding properties of Nikolskii spaces, it is enough to show
that $\bu \in N^{\suniq+\epsilon,p}(t_0,T;V_p)$, where
$$
\suniq = \frac{5}{2p}-1.
$$
To show this property, we use Lemma~\ref{lm:key1}. Hence, we need to check that (note that $\duniq=\suniq(p-1)$ is defined in \eqref{duniq})
\begin{equation}\label{needcheck}	
K_0(\bu) + \beh \in N^{\duniq+\epsilon,p'}(t_0,T;V_p')
\end{equation}
and that \eqref{ict0} holds true with
\begin{equation}
\tau> \frac{p\suniq}{2}=\frac{p' \duniq}{2}. \label{tuniq}
\end{equation}
Note that it is the assumption of Theorem~\ref{th:main} that $\beh \in N^{\duniq+\epsilon,p'}(t_0,T;V_p')$, so the second part of \eqref{needcheck}. Using the same assumption and combining it with Lemma~\ref{lm:w2}, we also obtain the validity of \eqref{ict0} with \eqref{tuniq}. Thus, we just need to check the first part of \eqref{needcheck}, i.e., the regularity of the convective term $K_0(\bu)$.

For this purpose, we use iteratively Lemmata~\ref{lm:key1}~and~\ref{lm:key2}. Notice that since we always will have that $\sigma$ appearing in Lemma~\ref{lm:key1} fulfills $\sigma \le \suniq$ then
$$
5p-13+6\sigma \le 5p-19+15/p <0
$$
for all $p\in [11/5,5/2)$. Hence we shall always use the first line in \eqref{conc:k2}.

We  distinguish
two cases.
\par
1. In case of $p\in (11/5,5/2)$ we will use an
iterative scheme. If $\bu \in N^{\sigma,p}(t_0,T;V_p)$, we improve regularity of the convective term by Lemma~\ref{lm:key2} and then use Lemma~\ref{lm:key1} to improve the regularity of $\bu$. More specifically,  we obtain $\bu \in
N^{\tilde{\sigma},p'}(t_0,T;V_p)$, where
\begin{equation}	\label{tsigma}
\tilde{\sigma} = \alpha + \beta \sigma,
\qquad \alpha = \frac{5p-11}{(p-1)(5p-6)},\
\beta = \frac{6}{(p-1)(5p-6)}.
\end{equation}
Note that if $\sigma=0$ we use the precise regularity of the convective term $K_0(\bu)$ from Lemma~\ref{lm:key2}. If $\sigma>0$ we can assume that also $\tau>0$ and we did  take the last term in \eqref{conc:k2},
namely the term with $\tau$, into account just to avoid the presence of $\epsilon$ in Lemma~\ref{lm:key2}. Inequality $p>11/5$ implies
that $\alpha > 0$ and $\beta\in(0,1)$.
Hence, the mapping $\sigma\to\tilde{\sigma}$ is a contraction on $[0,1]$. Banach contraction principle shows that starting from $\sigma = 0$, we can arrive arbitrarily
close to the fixed point $\sigma_{\rm max} = \alpha/(1-\beta) = 1/p$.
But obviously $\sigma_{\rm max} > \suniq$, so the proof is concluded after
finitely many steps.
\par
2. It remains to treat the critical case $p=11/5$. Observe that
now one has $\alpha=0$ and $\beta =1$ in \eqref{tsigma}, so
the previous iteration scheme no longer works. We modify the argument as
follows: by Lemma~\ref{lm:GHR}, the solution satisfies $\bu \in
L^{q}(t_0,T;V_p)$ with some $q>p$. Following now the argument of
Lemma~\ref{lm:key2}, we apply H\"older's inequality to \eqref{est-K-1}
with exponents $\tilde{P}$, $q/a$, $\infty$, $q/a$ and $p$. Since $q>p$,
one has $\tilde P < \infty$, and it follows that the convective term
$K_0(\bu)$ belongs to $N^{\delta,p'}(t_0,T;V_p')$ with some small
positive $\delta$.
\par
By Lemma~\ref{lm:key1}, the solution belongs to
$N^{\tau,\infty}(t_0,T;G)$ with some $\tau>0$.
Keeping this $\tau$, and combining now Lemmas~\ref{lm:key1} and
\ref{lm:key2}, while taking the last term \eqref{conc:k2} into
account, we obtain formula for improving $\sigma$ in the form
\begin{equation*}
    \tilde{\sigma} =  \frac{\tau(1-3\sigma)}{3} +\sigma = \sigma(1-\tau)+\frac{\tau}{3}.
\end{equation*}
Again the mapping $\sigma\to\tilde{\sigma}$ is a contraction on $[0,1]$ and by iterating the procedure we can get with $\sigma$ arbitrarily close to its fixed point $1/3$. Since $1/3>3/22=\suniq$ for $p=11/5$ we reach the value $\suniq$ after finitely many iterations.

We finish the proof by final comment about $t_0$. Since $\bu\in L^p(0,T; V_p)$ initially, we may chose an arbitrary Lebesgue point of $\bu(t)$ as $t_0$, which then can be used in Lemma~\ref{lm:w2}. Since almost every $t_0$ is the Lebesgue point of $\bu(t)$, we finally get the conclusion of Theorem~\ref{th:main} for all $t_0>0$. In addition, if $\bu_0 \in V_p$, we may set $t_0:=0$ and we again get the result of~Theorem~\ref{th:main}.
\end{proof}


\begin{proof}[Proof of Corollary~\ref{cor:uniq}.]

Using Lemma~\ref{lm:diff}, we see that it is enough to prove that $\bu \in L^{\puniq}(t_0,T;V_p)$ for some $t_0\in (0,\tau)$.
By Theorem~\ref{th:main}, both solutions have the regularity
\eqref{p:uniq} with some suitable $t_0 \in (0,\tau)$, and
the assertion of Corollary~\ref{cor:uniq} follows
from Lemma~\ref{lm:diff} and Gronwall's lemma.
\par
In the second part of the corollary the assumptions are chosen such that we can set $t_0=0$ by Theorem~\ref{th:main}.
\end{proof}

\begin{proof}[Proof of Corollary~\ref{cor:maxr}.]
By Theorem~\ref{th:main} and Lemma~\ref{lm:w2} we find $t_0\in (0,\tau)$ such that $\bu \in L^{\puniq}(t_0,T;V_p)$ and moreover that \eqref{nik:t0} holds true.  We conclude the proof by Lemma~\ref{lm:diff} with $\bu_1 = \th \bu$ and $\bu_2 = \bu$ and Gronwall's lemma.
  \par
In the second part of the corollary the assumptions are chosen such that we can set $t_0=0$ by Theorem~\ref{th:main}.
\end{proof}

\def\ocirc#1{\ifmmode\setbox0=\hbox{$#1$}\dimen0=\ht0 \advance\dimen0
  by1pt\rlap{\hbox to\wd0{\hss\raise\dimen0
  \hbox{\hskip.2em$\scriptscriptstyle\circ$}\hss}}#1\else {\accent"17 #1}\fi}
\providecommand{\bysame}{\leavevmode\hbox to3em{\hrulefill}\thinspace}
\providecommand{\MR}{\relax\ifhmode\unskip\space\fi MR }
\providecommand{\MRhref}[2]{%
  \href{http://www.ams.org/mathscinet-getitem?mr=#1}{#2}
}
\providecommand{\href}[2]{#2}


\end{document}